\def\bu{\bullet}
\def\marker{\>\hbox{${\vcenter{\vbox{
    \hrule height 0.4pt\hbox{\vrule width 0.4pt height 6pt
    \kern6pt\vrule width 0.4pt}\hrule height 0.4pt}}}$}\>}
\def\gpic#1{#1
     \smallskip\par\noindent{\centerline{\box\graph}} \medskip}

\def\bu{\bullet}
\def\marker{\>\hbox{${\vcenter{\vbox{
    \hrule height 0.4pt\hbox{\vrule width 0.4pt height 6pt
    \kern6pt\vrule width 0.4pt}\hrule height 0.4pt}}}$}\>}
\def\gpic#1{#1
     \smallskip\par\noindent{\centerline{\box\graph}} \medskip}
\documentclass[12pt]{article}
\usepackage{array,amsmath,amssymb,amsthm}  				    
\usepackage{fullpage,lineno}

\begin{document}


\input epsf.tex
\def\sp{\bigskip}
\def\ti{\\ \hglue \the \parindent}
\def\ce#1{\LP\medskip\centerline{#1}\medskip}
\def\LP{\par\noindent}

\newtheorem{theorem}{Theorem}[section]
\newtheorem{lemma}[theorem]{Lemma}
\newtheorem{corollary}[theorem]{Corollary}
\newtheorem{prop}[theorem]{Proposition}
\newtheorem{conj}[theorem]{Conjecture}
\newtheorem{claim}[theorem]{Claim}
\newtheorem{alg}[theorem]{Algorithm}
\theoremstyle{definition}
\newtheorem{remark}[theorem]{Remark}
\newtheorem{example}[theorem]{Example}
\newtheorem{definition}[theorem]{Definition}
\def\PF{\LP{\it Proof.}}
\def\qed{\ifhmode\unskip\nobreak\hfill$\Box$\bigskip\fi \ifmmode\eqno{Box}\fi}

\def\al{\alpha} \def\be{\beta}  \def\ga{\gamma} \def\dlt{\delta}
\def\eps{\epsilon} \def\th{\theta}  \def\ka{\kappa} \def\lmb{\lambda}
\def\sg{\sigma} \def\om{\omega}
\def\nul{\varnothing} 
\def\st{\colon\,}   
\def\MAP#1#2#3{#1\colon\,#2\to#3}
\def\VEC#1#2#3{#1_{#2},\ldots,#1_{#3}}
\def\VECOP#1#2#3#4{#1_{#2}#4\cdots #4 #1_{#3}}
\def\SE#1#2#3{\sum_{#1=#2}^{#3}}  \def\SGE#1#2{\sum_{#1\ge#2}}
\def\PE#1#2#3{\prod_{#1=#2}^{#3}} \def\PGE#1#2{\prod_{#1\ge#2}}
\def\UE#1#2#3{\bigcup_{#1=#2}^{#3}}
\def\CH#1#2{\binom{#1}{#2}} \def\MULT#1#2#3{\binom{#1}{#2,\ldots,#3}}
\def\FR#1#2{\frac{#1}{#2}}
\def\FL#1{\left\lfloor{#1}\right\rfloor} \def\FFR#1#2{\FL{\frac{#1}{#2}}}
\def\CL#1{\left\lceil{#1}\right\rceil}   \def\CFR#1#2{\CL{\frac{#1}{#2}}}
\def\Gb{\overline{G}}
\def\un#1{\underline{#1}}
\def\odd{$T$-odd}

\def\B#1{{\bf #1}}      \def\R#1{{\rm #1}}
\def\I#1{{\it #1}}      \def\c#1{{\cal #1}}
\def\C#1{\left | #1 \right |}    
\def\CC#1{\left \Vert #1 \right \Vert}    
\def\P#1{\left ( #1 \right )}    
\def\ov#1{\overline{#1}}        \def\un#1{\underline{#1}}

\def\aG{\al(G)} \def\aH{\al (H)} \def\tG{\th(G)} \def\tH{\th(H)}
\def\oG{\om(G)} \def\oH{\om(H)} \def\xG{\chi(G)} \def\xH{\chi (H)}
\def\Knnt{K_{\FL{n/2},\CL{n/2}}}
\def\Cnnt{\CH n{\FL{n/2}}}
\def\NN{{\Bbb N}} \def\ZZ{{\Bbb Z}} \def\QQ{{\Bbb Q}} \def\RR{{\Bbb R}}


\title{Cut-edges and regular factors\\ in regular graphs of odd degree}

\def\cutbd{\FL{\FR{2r+1}k}}
\def\esub{\subseteq}
\author{ 
Alexander V. Kostochka\thanks{University of Illinois at Urbana--Champaign,
Urbana IL 61801, and Sobolev Institute of Mathematics, Novosibirsk 630090,
Russia: \texttt{kostochk@math.uiuc.edu}.  Research supported in part by NSF
grants DMS-1600592 and grants 18-01-00353A and 16-01-00499  of the Russian
Foundation for Basic Research.}\,,
Andr\'e Raspaud\thanks{Universit\'{e} de Bordeaux, LaBRI UMR 5800,
F-33400 Talence, France: \texttt{raspaud@labri.fr}.}\,,
Bjarne Toft\thanks{University of Southern Denmark,
Odense, Denmark: \texttt{btoft@imada.sdu.dk}.}\,,\\
Douglas B. West\thanks{Zhejiang Normal University, Jinhua, China 321004
and University of Illinois at Urbana--Champaign, Urbana IL 61801:
\texttt{dwest@math.uiuc.edu}.
Research supported by Recruitment Program of Foreign Experts, 1000 Talent Plan,
State Administration of Foreign Experts Affairs, China.}\,,
Dara Zirlin\thanks{University of Illinois at Urbana--Champaign, Urbana IL 61801:
\texttt{zirlin2@illinois.edu}.}
}
\date{\today}
\maketitle

\begin{abstract}
We study $2k$-factors in $(2r+1)$-regular graphs.  Hanson, Loten, and Toft
proved that every $(2r+1)$-regular graph with at most $2r$ cut-edges has a
$2$-factor.  We generalize their result by proving for $k\le(2r+1)/3$ that
every $(2r+1)$-regular graph with at most $2r-3(k-1)$ cut-edges has a
$2k$-factor.  Both the restriction on $k$ and the restriction on the number of
cut-edges are sharp.  We characterize the graphs that have exactly
$2r-3(k-1)+1$ cut-edges but no $2k$-factor.  For $k>(2r+1)/3$, there are
graphs without cut-edges that have no $2k$-factor, as studied by 
Bollob\'as, Saito, and Wormald.
\end{abstract}

\section{Introduction}
An {\it $\ell$-factor} in a graph is an $\ell$-regular spanning subgraph.
In this paper we study the relationship between cut-edges and $2k$-factors in
regular graphs of odd degree.  In fact, all our results are for multigraphs,
allowing loops and multiedges, so the model we mean by ``graph'' allows
loops and multiedges.

The relationship between edge-connectivity and $1$-factors in regular graphs is
well known.  Petersen~\cite{Pete} proved that every $3$-regular graph with no
cut-edge decomposes into a $1$-factor and a $2$-factor, noting that the
conclusion also holds when all cut-edges lie along a path.
Sch\"onberger~\cite{Sch} proved that in a $3$-regular
graph with no cut-edge, every edge lies in some $1$-factor.
Berge~\cite{Ber} obtained the same conclusion for $r$-regular
$(r-1)$-edge-connected graphs of even order.
Finally, a result of
Plesn\'\i{k}~\cite{Ple} implies most of these statements: If $G$ is an
$r$-regular $(r-1)$-edge-connected multigraph with even order, and $G'$ is
obtained from $G$ by discarding at most $r-1$ edges, then $G'$ has a
$1$-factor.  The edge-connectivity condition is sharp: Katerinis~\cite{Kat}
determined the minimum number of vertices in an $r$-regular
$(r-2)$-edge-connected graph of even order having no $1$-factor.
Belck~\cite{Bel} and Bollob\'as, Saito, and Wormald~\cite{BSW} (independently)
determined all $(r,t,k)$ such that every $r$-regular $t$-edge-connected graph
has a $k$-factor; Niessen and Randerath~\cite{NR} further refined this in terms
of also the number of vertices.

Petersen was in fact more interested in $2$-factors.  The result about
$3$-regular graphs whose cut-edges lie on a path implies that every $3$-regular
graph with at most two cut-edges has a $2$-factor.  Also, there are $3$-regular
graphs with three cut-edges having no $2$-factor (communicated to Petersen by
Sylvester in 1889).  As a tool in a result about interval edge-coloring,
Hanson, Loten, and Toft~\cite{HLT} generalized Petersen's result to regular
graphs with larger odd degree.

\begin{theorem}[{\rm\cite{HLT}}]\label{HLT}
For $r\in\NN$, every $(2r+1)$-regular graph with at most $2r$ cut-edges
has a $2$-factor.
\end{theorem}

Petersen~\cite{Pete} also proved that every regular graph of even degree has a
$2$-factor.  Thus when $k\le r$ every $2r$-regular graph has a $2k$-factor.  
As a consequence, regular factors of degree $2k$ become harder to guarantee as
$k$ increases.  That is, a decomposition of a $(2r+1)$-regular graph into a
$2$-factor and $(2r-1)$-factor is easiest to find, while decomposition into a
$2r$-factor and $1$-factor is hardest to find (and implies the others).
 
In this paper, we generalize Theorem~\ref{HLT} to find the corresponding
best possible guarantee for $2k$-factors.  Limiting the number of cut-edges
suffices when $k$ is not too large.

\begin{theorem}\label{main}
For $r,k\in\NN$ with $k\le (2r+1)/3$, every $(2r+1)$-regular graph with at most
$2r-3(k-1)$ cut-edges has a $2k$-factor.  Furthermore, both inequalities are
sharp.
\end{theorem}

Earlier, Xiao and Liu~\cite{XL} proved a relationship between cut-edges
and $2k$-factors, showing that a $(2kr+s)$-regular graph with at most
$k(2r-3)+s$ cut-edges has a $2k$-factor avoiding any given edge.  Their
number of cut-edges in terms of degree and $k$ is similar to ours, since
$(2kr+s)-1-3(k-1)=k(2r-3)+s+2$, but their range of validity of $k$ in terms of
the degree of the full graph is more restricted than ours.

Our result is sharp in two ways.  First, when $k\le (2r+1)/3$ and there are
$2r+1-3(k-1)$ cut-edges, there may be no $2k$-factor.  Sylvester found examples
of such graphs (forbidding $2$-factors in a regular graph of odd degree 
greater than $1$ forbids all regular factors).  We complete the
Petersen--Sylvester investigation by describing all the extremal graphs without
$2k$-factors for general $k$.

\begin{theorem}
For $r,k\in\NN$ with $k\le (2r+1)/3$, a $(2r+1)$-regular graph with exactly
$2r+1-3(k-1)$ cut-edges fails to have a $2k$-factor if and only if it satisfies
the constructive structural description stated in Theorem~\ref{charzn}.
\end{theorem}

When $k>(2r+1)/3$, the condition in Theorem~\ref{main} cannot be satisfied,
and in fact there are $(2r+1)$-regular graphs that have no $2k$-factor even
though they have no cut-edges.  A $2k$-factor can instead be guaranteed by
edge-connectivity requirements.  The result of Berge~\cite{Ber} implies that
$(2r+1)$-regular $2r$-edge-connected graphs have $1$-factors and hence
factors of all even degrees, by the $2$-factor theorem of Petersen~\cite{Pete}.
Therefore, when $k>(2r+1)/3$ the natural question becomes what
edge-connectivity suffices to guarantee a $2k$-factor.

As mentioned earlier, this problem was solved by Bollob\'as, Saito, and
Wormald~\cite{BSW}, who determined all triples $(r,t,k)$ such that every
$r$-regular $t$-edge-connected multigraph has a $k$-factor (the triples are the
same for simple graphs).  As noted by H\"aggkvist~\cite{Hag} and by
Niessen and Randerath~\cite{NR}, earlier Belck~\cite{Bel} obtained the
result (in 1950).  Earlier still, Baebler~\cite{Bab} proved the weaker result
that $2k$-edge-connected $(2r+1)$-regular graphs have $2k$-factors.

The special case of the result of~\cite{BSW} that applies here (even-regular
factors of odd-regular multigraphs) is that all ${(2r+1)}$-regular
$2t$-edge-connected or $(2t+1)$-edge-connected multigraphs have $2k$-factors if
and only if ${k\le \frac{t}{2t+1}}(2r+1)$.  The general construction given
in~\cite{BSW}, which covers additional cases, is quite complicated.  Here we
provide a very simple construction that completes our investigation and shows
necessity of their condition for even-regular factors of odd-regular graphs.
In particular, for $1\le t<r$ and $k>\frac{t}{2t+1}(2r+1)$ we present an
easily described $(2t+1)$-connected simple graph that has no $2k$-factor.

Our results use the necessary and sufficient condition for the existence of
$\ell$-factors that was initially proved by Belck~\cite{Bel} and is a special
case of the $f$-Factor Theorem of Tutte~\cite{Tut1,Tut2}.  When $T$ is a set of
vertices in a graph $G$, let $d_G(T)=\sum_{v\in T} d_G(v)$, where $d_G(v)$ is
the degree of $v$ in $G$.  With $|T|$ for the size of a vertex set $T$, we also
write $\CC T$ for the number of edges induced by $T$ and $\CC{A,B}$ for the
number of edges having endpoints in both $A$ and $B$ (when $A\cap B=\nul$).
The characterization is the following.

\begin{theorem}[\cite{Bel,Tut1,Tut2}]\label{lfact}
A multigraph $G$ has a $\ell$-factor if and only if 
\begin{equation}\label{lfac}
q(S,T)-d_{G-S}(T)\le \ell(\C S-\C T)
\end{equation}
for all disjoint subsets $S,T\subset V(G)$, where $q(S,T)$ is the number of
components $Q$ of $G-S-T$ such that $\CC{V(Q),T}+\ell\C{V(Q)}$ is odd.
\end{theorem}

Since we consider only the situation where $\ell=2k$, the criterion for
a component $Q$ of $G-S-T$ to be counted by $q(S,T)$ simplifies to
$\CC{V(Q),T}$ being odd.

\section{Cut-edges and $2k$-factors}

In this section we generalize Theorem~\ref{HLT} to $2k$-factors. 
\begin{theorem}\label{2kfac}
For $r,k\in\NN$ with $k\le(2r+1)/3$, every $(2r+1)$-regular multigraph with at
most $2r-3(k-1)$ cut-edges has a $2k$-factor.
\end{theorem}

\begin{proof}
Let $G$ be a $(2r+1)$-regular multigraph having no $2k$-factor, and let $p$ be
the number of cut-edges in $G$.  We prove $p>2r-3(k-1)$.  By setting
$\ell=2k$ in Theorem~\ref{lfact}, lack of a $2k$-factor requires disjoint sets
$S,T\subseteq V(G)$ such that $q(S,T)> 2k(\C S-\C T)+d_{G-S}(T)$.

Letting $R=V(G)-S-T$,  the quantity $q(S,T)$ becomes the number
of components $Q$ of $G[R]$ such that $\CC{V(Q),T}$ is odd.  Thus $q(S,T)$ has
the same parity as $\CC{R,T}$.  In turn, $\CC{R,T}$ has the same parity as
$d_{G-S}(T)$, since the latter counts edges from $R$ to $T$ once and edges
within $T$ twice.  Hence the two sides of the inequality above have the same
parity.  We conclude
\begin{equation}\label{e1}
q(S,T)\geq d_{G-S}(T)+2k(|S|-|T|)+2.
\end{equation} 
 
Say that a subgraph $H$ of $G-T$ is {\it \odd} if $\CC{V(H),T}$ is odd.  The
components of $G-S-T$ that are \odd\ are the components counted by $q(S,T)$.
Each \odd\ component contributes at least $1$ to $d_{G-S}(T)$.  Hence
(\ref{e1}) cannot hold with $|S|\geq |T|$, and we may assume $|T|>|S|$.

Let $q_1$ be the number of \odd\ components having one edge to $T$ and no edges
to $S$; since that edge is a cut-edge, $q_1\le p$.
Let $q_2$ be the number of \odd\ components having one edge to $T$ and at least
one edge to $S$; note that $q_2\le \CC{R,S}$.  
Let $q_3$ be the number of \odd\ components having at least three edges to $T$;
thus $q_1+q_2+3q_3\le d_{G-S}(T)$.  Note also that $q(S,T)=q_1+q_2+q_3$.
Summing the last inequality with two copies of the first two yields
$$3q(S,T)=3(q_1+q_2+q_3)\leq 2p+2\CC{R,S}+d_{G-S}(T).$$
Combining this inequality with (\ref{e1}) yields
$$
{2p+2\CC{R,S}+ d_{G-S}(T)}\geq 3 d_{G-S}(T)+6k(|S|-|T|)+6,
$$ 
which simplifies to
\begin{equation}\label{e5}
\CC{R,S}\geq 3-p+ d_{G-S}(T)+3k(|S|-|T|).
\end{equation}

On the other hand, since $G$ is $(2r+1)$-regular,
$$d_{G-S}(T)=(2r+1)|T|-\CC{T,S}\geq (2r+1)|T|-\left[(2r+1)|S|-\CC{R,S}\right].
$$
Using this inequality, (\ref{e5}), and $|T|-|S|\geq 1$, the given hypothesis
$2r+1-3k \ge 0$ yields
$$\CC{R,S}\geq 3-p+(2r+1-3k)(|T|-|S|)+\CC{R,S}\geq 3-p+(2r+1-3k)+\CC{R,S}.$$
This simplifies to $p\geq 2r+1-3(k-1)$, as claimed.
\end{proof}

\section{Fewest cut-edges with no $2k$-factor}

To describe the extremal graphs, we begin with a definition.
Keep in mind that here ``graph'' allows loops and multiedges.

\begin{definition}
In a $(2r+1)$-regular graph $G$, the result of {\it blistering} an edge
$e\in E(G)$ by a $(2r+1)$-regular graph $H$ having no cut-edge is a graph $G'$
obtained from the disjoint union $G+H$ by deleting $e$ and an edge $e'\in E(H)$
(where $e'$ may be a loop if $r>1$), followed by adding two disjoint edges to
make each endpoint of $e$ adjacent to one endpoint of $e'$.  The resulting
graph $G'$ is $(2r+1)$-regular.
\end{definition}

Figure~\ref{k1fig} illustrates blistering of one edge joining $S$ and $T$
in a $3$-regular graph $G$ with three cut-edges and no $2$-factor to obtain a
larger such graph $G'$.  The components of $G'-S-T$ labeled $Q_i$ are
components counted by $q_i$, for $i\in\{1,2,3\}$.

\begin{figure}[h]
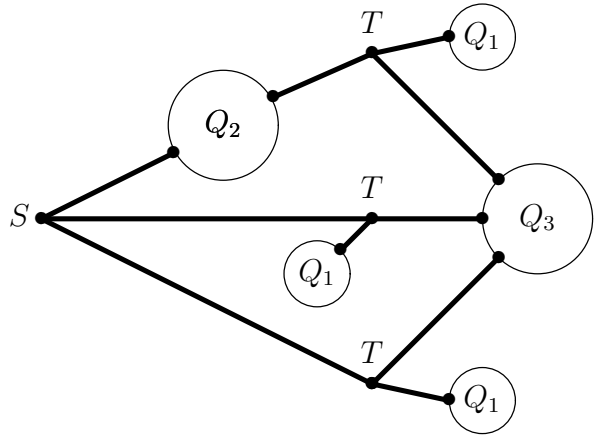

\gpic{
\expandafter\ifx\csname graph\endcsname\relax \csname newbox\endcsname\graph\fi
\expandafter\ifx\csname graphtemp\endcsname\relax \csname newdimen\endcsname\graphtemp\fi
\setbox\graph=\vtop{\vskip 0pt\hbox{%
    \special{pn 8}%
    \special{ar 2423 2077 173 173 0 6.28319}%
    \special{ar 1558 1413 173 173 0 6.28319}%
    \special{ar 2423 173 173 173 0 6.28319}%
    \special{ar 2712 1125 288 288 0 6.28319}%
    \special{ar 1067 635 288 288 0 6.28319}%
    \graphtemp=.5ex\advance\graphtemp by 1.125in
    \rlap{\kern 0.115in\lower\graphtemp\hbox to 0pt{\hss $\bu$\hss}}%
    \graphtemp=.5ex\advance\graphtemp by 1.990in
    \rlap{\kern 1.846in\lower\graphtemp\hbox to 0pt{\hss $\bu$\hss}}%
    \graphtemp=.5ex\advance\graphtemp by 0.781in
    \rlap{\kern 0.806in\lower\graphtemp\hbox to 0pt{\hss $\bu$\hss}}%
    \graphtemp=.5ex\advance\graphtemp by 0.488in
    \rlap{\kern 1.329in\lower\graphtemp\hbox to 0pt{\hss $\bu$\hss}}%
    \graphtemp=.5ex\advance\graphtemp by 1.990in
    \rlap{\kern 1.846in\lower\graphtemp\hbox to 0pt{\hss $\bu$\hss}}%
    \graphtemp=.5ex\advance\graphtemp by 2.077in
    \rlap{\kern 2.250in\lower\graphtemp\hbox to 0pt{\hss $\bu$\hss}}%
    \graphtemp=.5ex\advance\graphtemp by 1.125in
    \rlap{\kern 1.846in\lower\graphtemp\hbox to 0pt{\hss $\bu$\hss}}%
    \graphtemp=.5ex\advance\graphtemp by 0.173in
    \rlap{\kern 2.250in\lower\graphtemp\hbox to 0pt{\hss $\bu$\hss}}%
    \graphtemp=.5ex\advance\graphtemp by 0.260in
    \rlap{\kern 1.846in\lower\graphtemp\hbox to 0pt{\hss $\bu$\hss}}%
    \special{pn 28}%
    \special{pa 115 1125}%
    \special{pa 806 781}%
    \special{fp}%
    \special{pa 115 1125}%
    \special{pa 1846 1990}%
    \special{fp}%
    \special{pa 115 1125}%
    \special{pa 1846 1125}%
    \special{fp}%
    \special{pa 1846 1990}%
    \special{pa 1846 1990}%
    \special{fp}%
    \special{pa 1329 488}%
    \special{pa 1846 260}%
    \special{fp}%
    \special{pa 1846 1990}%
    \special{pa 2250 2077}%
    \special{fp}%
    \special{pa 2250 173}%
    \special{pa 1846 260}%
    \special{fp}%
    \special{pa 1846 1125}%
    \special{pa 1680 1291}%
    \special{fp}%
    \special{pa 2508 1329}%
    \special{pa 1846 1990}%
    \special{fp}%
    \special{pa 2423 1125}%
    \special{pa 1846 1125}%
    \special{fp}%
    \special{pa 2508 921}%
    \special{pa 1846 260}%
    \special{fp}%
    \graphtemp=.5ex\advance\graphtemp by 1.125in
    \rlap{\kern 0.000in\lower\graphtemp\hbox to 0pt{\hss $S$\hss}}%
    \graphtemp=.5ex\advance\graphtemp by 0.635in
    \rlap{\kern 1.067in\lower\graphtemp\hbox to 0pt{\hss $Q_2$\hss}}%
    \graphtemp=.5ex\advance\graphtemp by 0.635in
    \rlap{\kern 1.067in\lower\graphtemp\hbox to 0pt{\hss $Q_2$\hss}}%
    \graphtemp=.5ex\advance\graphtemp by 1.846in
    \rlap{\kern 1.846in\lower\graphtemp\hbox to 0pt{\hss $T$\hss}}%
    \graphtemp=.5ex\advance\graphtemp by 0.981in
    \rlap{\kern 1.846in\lower\graphtemp\hbox to 0pt{\hss $T$\hss}}%
    \graphtemp=.5ex\advance\graphtemp by 0.115in
    \rlap{\kern 1.846in\lower\graphtemp\hbox to 0pt{\hss $T$\hss}}%
    \graphtemp=.5ex\advance\graphtemp by 2.077in
    \rlap{\kern 2.423in\lower\graphtemp\hbox to 0pt{\hss $Q_1$\hss}}%
    \graphtemp=.5ex\advance\graphtemp by 1.413in
    \rlap{\kern 1.558in\lower\graphtemp\hbox to 0pt{\hss $Q_1$\hss}}%
    \graphtemp=.5ex\advance\graphtemp by 0.173in
    \rlap{\kern 2.423in\lower\graphtemp\hbox to 0pt{\hss $Q_1$\hss}}%
    \graphtemp=.5ex\advance\graphtemp by 1.125in
    \rlap{\kern 2.712in\lower\graphtemp\hbox to 0pt{\hss $Q_3$\hss}}%
    \graphtemp=.5ex\advance\graphtemp by 1.291in
    \rlap{\kern 1.680in\lower\graphtemp\hbox to 0pt{\hss $\bu$\hss}}%
    \graphtemp=.5ex\advance\graphtemp by 1.329in
    \rlap{\kern 2.508in\lower\graphtemp\hbox to 0pt{\hss $\bu$\hss}}%
    \graphtemp=.5ex\advance\graphtemp by 1.125in
    \rlap{\kern 2.423in\lower\graphtemp\hbox to 0pt{\hss $\bu$\hss}}%
    \graphtemp=.5ex\advance\graphtemp by 0.921in
    \rlap{\kern 2.508in\lower\graphtemp\hbox to 0pt{\hss $\bu$\hss}}%
    \hbox{\vrule depth2.250in width0pt height 0pt}%
    \kern 3.000in
  }%
}%
}
\vspace{-1pc}
\caption{A class of $3$-regular graphs with three cut-edges and no
$2$-factor.\label{k1fig}}
\end{figure}

\begin{theorem}\label{charzn}
For $k\le(2r+1)/3$, a $(2r+1)$-regular graph with $2r+4-3k$ cut-edges has no
$2k$-factor if and only if the vertex set $V(G)$ has a partition into sets
$R,S,T$ such that \\
(a) $S$ and $T$ are independent sets with $|T|>|S|$,\\
(b) all cut-edges join $T$ to distinct components of $G[R]$,\\
(c) all edges incident to $S$ lead to $T$ (possibly via blisters that are
components of $G[R]$),\\
(d) exactly $k(|T|-|S|)-1$ components of $G[R]$ are joined to $T$ by exactly
three edges each,\\
(e) each remaining component of $R$ is $(2r+1)$-regular, with no cut-edge,
and \\
(f) if $k<(2r+1)/3$, then $|T|-|S|=1$.
\looseness-1
\end{theorem}

\begin{proof}
{\it Sufficiency:}  Let $G$ be a graph $G$ with $2r+4-3k$ cut-edges, and
suppose that such a partition $\{R,S,T\}$ of $V(G)$ exists.  Let $q_2$ be the
number of components of $G[R]$ that blister edges from $S$ to $T$.  Each
cut-edge joins $T$ to a \odd\ component, by (b).  The $k(\C T-\C S)-1$
components of $G[R]$ joined to $T$ by three edges (according to (d)) are also
\odd, as are the $q_2$ components of $G[R]$ arising as blisters.  Hence
$q(S,T)\ge2r+4-3k+k(\C T-\C S)-1+q_2$.  The number of edges joining $S$ and $T$
is $(2r+1)\C S-q_2$, by (c).  Using also (a), we have
$d_{G-S}(T)=(2r+1)(\C T-\C S)+q_2$.  We compute
\begin{align*}
q(S,T)-d_{G-S}(T)&\ge(2r+1-3k)+2+(k-2r-1)(\C T-\C S)\\
&=-(2r+1-3k)(\C T-\C S-1)+2k(\C S-\C T)+2= 2k(\C S-\C T)+2,
\end{align*}
where the last equality uses (f) and the restriction $k\le(2r+1)/3$.  Hence
the given partition $R,S,T$ satisfies (\ref{e1}), and $G$ has no $2k$-factor.

{\it Necessity:}
Suppose that $G$ has $2r+1-3(k-1)$ cut-edges and no $2k$-factor; we obtain the
described partition of $V(G)$. The proof of Theorem~\ref{2kfac} considers
$(2r+1)$-regular graphs with no $2k$-factor and produces $p\ge 2r+4-3k$, where
$p$ is the number of cut-edges.  To avoid having more cut-edges, we must have
equality in all the inequalities used to produce this lower bound.

Recall that $q(S,T)$ counts the components $Q$ of $G[R]$ with $\CC{V(Q),T}$
odd.  Also $q(S,T)=q_1+q_2+q_3$, where $q_1,q_2,q_3$ count the components
having one edge to $T$ and none to $S$, one edge to $T$ and at least one to
$S$, and at least three edges to $T$, respectively.  Equality in the
computation of Theorem~\ref{2kfac} requires all of the following.
\begin{equation}\label{e6}
q_1=p
\end{equation}
\begin{equation}\label{e8}
q_2=\CC{R,S}
\end{equation}
\begin{equation}\label{e7}
q_1+q_2+3q_3=d_{G-S}(T)
\end{equation}
\begin{equation}\label{e9}
(2r+1)|S|=\CC{T,S}+\CC{R,S}
\end{equation}
\begin{equation}\label{e10}
|T|-|S|\ge1, \textrm{with equality when $k<(2r+1)/3$}
\end{equation}

By (\ref{e7}), contributions to $d_G(T)$ not in $\CC{T,S}$ are counted in
$\CC{T,R}$, so $T$ is independent.  By (\ref{e9}), all edges incident to $S$
lead to $T$ or $R$, so $S$ is independent, proving (a).  The first observation
in proving Theorem~\ref{2kfac} was $|T|>|S|$, and equality in the last step
requires $|T|-|S|=1$ when $2r+1>3k$, as stated in (\ref{e10}) and desired in
(f).  By $(\ref{e6})$, the cut-edges join $T$ to distinct components of $G[R]$,
proving (b).

By (\ref{e8}) and (\ref{e9}), $q_2=0$ implies $(2r+1)|S|=\CC{T,S}$,
making all edges incident to $S$ incident also to $T$.
Since $(2r+1)|S|=\CC{T,S}+q_2$, each component of $G[R]$ counted by $q_2$
generates only one edge from $R$ to $S$.  Thus each such component blisters an
edge joining $S$ and $T$ in a smaller such graph.  This explains all the
edges counted by $\CC{S,R}$.  Hence we view the edges incident to $S$ as
edges to $T$ with possible blisters, proving (c).
 
We have accounted for $(2r+1)|S|$ edges incident to $T$ leading to $S$,
including through $q_2$ blisters.  There are also $p$ cut-edges leading to
components of $G[R]$, where $p=2r+1-3(k-1)$.  This leaves
$(2r+1)|T|-(2r+1)+3(k-1)-(2r+1)|S|$ edges incident to $T$ that are not
cut-edges and join $T$ to components of $G[R]$ not counted by $q_2$.

When $k<(2r+1)/3$ and $|T|-|S|=1$, this expression simplifies to $3(k-1)$.
When $k=(2r+1)/3$, it simplifies to $3[k(|T|-|S|)-1]$, which is valid for
both cases.  By (\ref{e7}), all remaining edges incident to $T$ connect
vertices of $T$ to \odd\ components of $G[R]$ counted by $q_3$, using exactly
three edges for each such component.  Hence there are exactly $k(|T|-|S|)-1$
such components of $G[R]$, proving (d).  This completes the description of the
\odd\ components.

Since we have described all edges incident to $S$ and $T$, any remaining
components of $G[R]$ are actually $(2r+1)$-regular components of $G$ without
cut-edges, proving (e).  They do not affect the number of \odd\ components or
the existence of a $2k$-factor.
\end{proof}

Theorem~\ref{charzn} can be viewed as a constructive procedure for generating
all extremal examples from certain base graphs.  Given $r$ and $k$ with
$k\le (2r+1)/3$, we start with a bipartite graph having parts $T$ and
$R\cup S$, where $\C T-\C S\ge1$, with equality if $k<(2r+1)/3$.  Also,
vertices in $T\cup S$ have degree $2r+1$, and $R$ has $2r+4-3k$ vertices of
degree $1$ and $k(\C T-\C S)-1$ vertices of degree $3$.  We expand the vertices
of $R$ to obtain a $(2r+1)$-regular multigraph $G$.  This is a base graph.
We can then blister edges from $S$ to $T$ and/or add $(2r+1)$-regular
$2$-edge-connected components.

The case $|T|=1$ and $|S|=0$ gives the graphs found by Sylvester.
When $k>(2r+1)/3$, an inequality used in the proof of Theorem~\ref{2kfac} is
not valid.  In this range no restriction on cut-edges can guarantee a
$2k$-factor; we present a simple general construction.  As mentioned earlier,
this is a sharpness example for the result of Bollob\'as, Saito, and
Wormald~\cite{BSW} that every ${(2r+1)}$-regular $2t$-edge-connected or
$(2t+1)$-edge-connected multigraph has a $2k$-factor if and only if
${k\le \frac{t}{2t+1}}(2r+1)$.  It is simpler than their more general
construction.

\begin{theorem}\label{highk}
For $1\le t< r$ and $k>\frac{t}{2t+1}(2r+1)$, there is a $(2t+1)$-connected
$(2r+1)$-regular graph having no $2k$-factor.
\end{theorem}

\begin{proof}
Let $H_{r,t}$ be the complement of $C_{2t+1}+(r-t+1)K_2$.  That is, $H_{r,t}$
is obtained from the complete graph $K_{2r+3}$ by deleting the edges of a
$(2t+1)$-cycle and $r-t+1$ other pairwise disjoint edges not incident to the
cycle.  Note that in $H_{r,t}$ the vertices of the deleted cycle have degree
$2r$, while the remaining vertices have degree $2r+1$.  Let $G$ be the graph
formed from the disjoint union of $2r+1$ copies of $H_{r,t}$ by adding a set
$T$ of $2t+1$ vertices and $2r+1$ matchings joining $T$ to the vertices
of the deleted cycle in each copy of $H_{r,t}$ (see Figure~\ref{fighighk}).

Deleting $2t$ vertices cannot separate any copy of $H_{r,t}$ from $T$, and
any two vertices of $T$ are connected by $2r+1$ disjoint paths through the
copies of $H_{r,t}$, so $G$ is $(2t+1)$-connected.

Suppose that $G$ has a $2k$-factor $F$.  Every edge cut in an even factor
is crossed by an even number of edges, since the factor decomposes into cycles.
Hence $F$ has at most $2t$ edges joining $T$ to each copy of $H_{r,t}$.
On the other hand, since $T$ is independent, $F$ must have $2k|T|$ edges
leaving $T$.  Thus $2k(2t+1)\le 2t(2r+1)$.
%
\end{proof}




\begin{figure}[h]
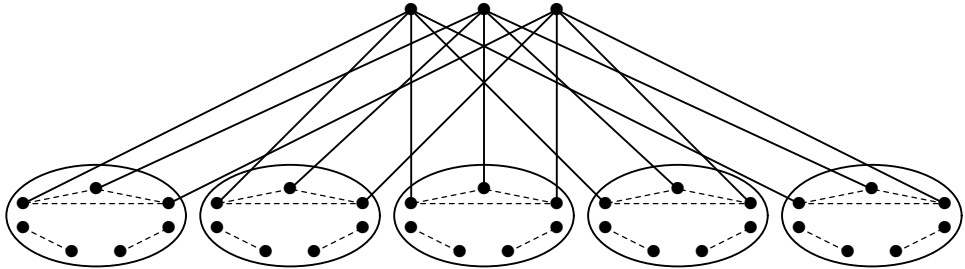

\gpic{
\expandafter\ifx\csname graph\endcsname\relax \csname newbox\endcsname\graph\fi
\expandafter\ifx\csname graphtemp\endcsname\relax \csname newdimen\endcsname\graphtemp\fi
\setbox\graph=\vtop{\vskip 0pt\hbox{%
    \graphtemp=.5ex\advance\graphtemp by 0.032in
    \rlap{\kern 2.119in\lower\graphtemp\hbox to 0pt{\hss $\bu$\hss}}%
    \graphtemp=.5ex\advance\graphtemp by 0.032in
    \rlap{\kern 2.500in\lower\graphtemp\hbox to 0pt{\hss $\bu$\hss}}%
    \graphtemp=.5ex\advance\graphtemp by 0.032in
    \rlap{\kern 2.881in\lower\graphtemp\hbox to 0pt{\hss $\bu$\hss}}%
    \graphtemp=.5ex\advance\graphtemp by 1.174in
    \rlap{\kern 0.089in\lower\graphtemp\hbox to 0pt{\hss $\bu$\hss}}%
    \graphtemp=.5ex\advance\graphtemp by 1.301in
    \rlap{\kern 0.343in\lower\graphtemp\hbox to 0pt{\hss $\bu$\hss}}%
    \graphtemp=.5ex\advance\graphtemp by 1.301in
    \rlap{\kern 0.596in\lower\graphtemp\hbox to 0pt{\hss $\bu$\hss}}%
    \graphtemp=.5ex\advance\graphtemp by 1.174in
    \rlap{\kern 0.850in\lower\graphtemp\hbox to 0pt{\hss $\bu$\hss}}%
    \graphtemp=.5ex\advance\graphtemp by 1.047in
    \rlap{\kern 0.850in\lower\graphtemp\hbox to 0pt{\hss $\bu$\hss}}%
    \graphtemp=.5ex\advance\graphtemp by 0.971in
    \rlap{\kern 0.470in\lower\graphtemp\hbox to 0pt{\hss $\bu$\hss}}%
    \graphtemp=.5ex\advance\graphtemp by 1.047in
    \rlap{\kern 0.089in\lower\graphtemp\hbox to 0pt{\hss $\bu$\hss}}%
    \special{pn 11}%
    \special{pa 2119 32}%
    \special{pa 89 1047}%
    \special{fp}%
    \special{pa 2500 32}%
    \special{pa 470 971}%
    \special{fp}%
    \special{pa 2881 32}%
    \special{pa 850 1047}%
    \special{fp}%
    \special{pn 8}%
    \special{pa 89 1174}%
    \special{pa 343 1301}%
    \special{da 0.025}%
    \special{pa 596 1301}%
    \special{pa 850 1174}%
    \special{da 0.025}%
    \special{pa 850 1047}%
    \special{pa 470 971}%
    \special{pa 89 1047}%
    \special{pa 850 1047}%
    \special{da 0.025}%
    \special{pn 11}%
    \special{ar 470 1110 470 266 0 6.28319}%
    \graphtemp=.5ex\advance\graphtemp by 1.174in
    \rlap{\kern 1.104in\lower\graphtemp\hbox to 0pt{\hss $\bu$\hss}}%
    \graphtemp=.5ex\advance\graphtemp by 1.301in
    \rlap{\kern 1.358in\lower\graphtemp\hbox to 0pt{\hss $\bu$\hss}}%
    \graphtemp=.5ex\advance\graphtemp by 1.301in
    \rlap{\kern 1.612in\lower\graphtemp\hbox to 0pt{\hss $\bu$\hss}}%
    \graphtemp=.5ex\advance\graphtemp by 1.174in
    \rlap{\kern 1.865in\lower\graphtemp\hbox to 0pt{\hss $\bu$\hss}}%
    \graphtemp=.5ex\advance\graphtemp by 1.047in
    \rlap{\kern 1.865in\lower\graphtemp\hbox to 0pt{\hss $\bu$\hss}}%
    \graphtemp=.5ex\advance\graphtemp by 0.971in
    \rlap{\kern 1.485in\lower\graphtemp\hbox to 0pt{\hss $\bu$\hss}}%
    \graphtemp=.5ex\advance\graphtemp by 1.047in
    \rlap{\kern 1.104in\lower\graphtemp\hbox to 0pt{\hss $\bu$\hss}}%
    \special{pa 2119 32}%
    \special{pa 1104 1047}%
    \special{fp}%
    \special{pa 2500 32}%
    \special{pa 1485 971}%
    \special{fp}%
    \special{pa 2881 32}%
    \special{pa 1865 1047}%
    \special{fp}%
    \special{pn 8}%
    \special{pa 1104 1174}%
    \special{pa 1358 1301}%
    \special{da 0.025}%
    \special{pa 1612 1301}%
    \special{pa 1865 1174}%
    \special{da 0.025}%
    \special{pa 1865 1047}%
    \special{pa 1485 971}%
    \special{pa 1104 1047}%
    \special{pa 1865 1047}%
    \special{da 0.025}%
    \special{pn 11}%
    \special{ar 1485 1110 470 266 0 6.28319}%
    \graphtemp=.5ex\advance\graphtemp by 1.174in
    \rlap{\kern 2.119in\lower\graphtemp\hbox to 0pt{\hss $\bu$\hss}}%
    \graphtemp=.5ex\advance\graphtemp by 1.301in
    \rlap{\kern 2.373in\lower\graphtemp\hbox to 0pt{\hss $\bu$\hss}}%
    \graphtemp=.5ex\advance\graphtemp by 1.301in
    \rlap{\kern 2.627in\lower\graphtemp\hbox to 0pt{\hss $\bu$\hss}}%
    \graphtemp=.5ex\advance\graphtemp by 1.174in
    \rlap{\kern 2.881in\lower\graphtemp\hbox to 0pt{\hss $\bu$\hss}}%
    \graphtemp=.5ex\advance\graphtemp by 1.047in
    \rlap{\kern 2.881in\lower\graphtemp\hbox to 0pt{\hss $\bu$\hss}}%
    \graphtemp=.5ex\advance\graphtemp by 0.971in
    \rlap{\kern 2.500in\lower\graphtemp\hbox to 0pt{\hss $\bu$\hss}}%
    \graphtemp=.5ex\advance\graphtemp by 1.047in
    \rlap{\kern 2.119in\lower\graphtemp\hbox to 0pt{\hss $\bu$\hss}}%
    \special{pa 2119 32}%
    \special{pa 2119 1047}%
    \special{fp}%
    \special{pa 2500 32}%
    \special{pa 2500 971}%
    \special{fp}%
    \special{pa 2881 32}%
    \special{pa 2881 1047}%
    \special{fp}%
    \special{pn 8}%
    \special{pa 2119 1174}%
    \special{pa 2373 1301}%
    \special{da 0.025}%
    \special{pa 2627 1301}%
    \special{pa 2881 1174}%
    \special{da 0.025}%
    \special{pa 2881 1047}%
    \special{pa 2500 971}%
    \special{pa 2119 1047}%
    \special{pa 2881 1047}%
    \special{da 0.025}%
    \special{pn 11}%
    \special{ar 2500 1110 470 266 0 6.28319}%
    \graphtemp=.5ex\advance\graphtemp by 1.174in
    \rlap{\kern 3.135in\lower\graphtemp\hbox to 0pt{\hss $\bu$\hss}}%
    \graphtemp=.5ex\advance\graphtemp by 1.301in
    \rlap{\kern 3.388in\lower\graphtemp\hbox to 0pt{\hss $\bu$\hss}}%
    \graphtemp=.5ex\advance\graphtemp by 1.301in
    \rlap{\kern 3.642in\lower\graphtemp\hbox to 0pt{\hss $\bu$\hss}}%
    \graphtemp=.5ex\advance\graphtemp by 1.174in
    \rlap{\kern 3.896in\lower\graphtemp\hbox to 0pt{\hss $\bu$\hss}}%
    \graphtemp=.5ex\advance\graphtemp by 1.047in
    \rlap{\kern 3.896in\lower\graphtemp\hbox to 0pt{\hss $\bu$\hss}}%
    \graphtemp=.5ex\advance\graphtemp by 0.971in
    \rlap{\kern 3.515in\lower\graphtemp\hbox to 0pt{\hss $\bu$\hss}}%
    \graphtemp=.5ex\advance\graphtemp by 1.047in
    \rlap{\kern 3.135in\lower\graphtemp\hbox to 0pt{\hss $\bu$\hss}}%
    \special{pa 2119 32}%
    \special{pa 3135 1047}%
    \special{fp}%
    \special{pa 2500 32}%
    \special{pa 3515 971}%
    \special{fp}%
    \special{pa 2881 32}%
    \special{pa 3896 1047}%
    \special{fp}%
    \special{pn 8}%
    \special{pa 3135 1174}%
    \special{pa 3388 1301}%
    \special{da 0.025}%
    \special{pa 3642 1301}%
    \special{pa 3896 1174}%
    \special{da 0.025}%
    \special{pa 3896 1047}%
    \special{pa 3515 971}%
    \special{pa 3135 1047}%
    \special{pa 3896 1047}%
    \special{da 0.025}%
    \special{pn 11}%
    \special{ar 3515 1110 470 266 0 6.28319}%
    \graphtemp=.5ex\advance\graphtemp by 1.174in
    \rlap{\kern 4.150in\lower\graphtemp\hbox to 0pt{\hss $\bu$\hss}}%
    \graphtemp=.5ex\advance\graphtemp by 1.301in
    \rlap{\kern 4.404in\lower\graphtemp\hbox to 0pt{\hss $\bu$\hss}}%
    \graphtemp=.5ex\advance\graphtemp by 1.301in
    \rlap{\kern 4.657in\lower\graphtemp\hbox to 0pt{\hss $\bu$\hss}}%
    \graphtemp=.5ex\advance\graphtemp by 1.174in
    \rlap{\kern 4.911in\lower\graphtemp\hbox to 0pt{\hss $\bu$\hss}}%
    \graphtemp=.5ex\advance\graphtemp by 1.047in
    \rlap{\kern 4.911in\lower\graphtemp\hbox to 0pt{\hss $\bu$\hss}}%
    \graphtemp=.5ex\advance\graphtemp by 0.971in
    \rlap{\kern 4.530in\lower\graphtemp\hbox to 0pt{\hss $\bu$\hss}}%
    \graphtemp=.5ex\advance\graphtemp by 1.047in
    \rlap{\kern 4.150in\lower\graphtemp\hbox to 0pt{\hss $\bu$\hss}}%
    \special{pa 2119 32}%
    \special{pa 4150 1047}%
    \special{fp}%
    \special{pa 2500 32}%
    \special{pa 4530 971}%
    \special{fp}%
    \special{pa 2881 32}%
    \special{pa 4911 1047}%
    \special{fp}%
    \special{pn 8}%
    \special{pa 4150 1174}%
    \special{pa 4404 1301}%
    \special{da 0.025}%
    \special{pa 4657 1301}%
    \special{pa 4911 1174}%
    \special{da 0.025}%
    \special{pa 4911 1047}%
    \special{pa 4530 971}%
    \special{pa 4150 1047}%
    \special{pa 4911 1047}%
    \special{da 0.025}%
    \special{pn 11}%
    \special{ar 4530 1110 470 266 0 6.28319}%
    \hbox{\vrule depth1.377in width0pt height 0pt}%
    \kern 5.000in
  }%
}%
}

\vspace{-1pc}
\caption{$(2r+1)$-regular, $(2t+1)$-connected, no $2k$-factor
($(r,t,k)=(2,1,2)$ shown).\label{fighighk}}
\end{figure}

\end{document}